\documentclass[11pt]{article}

\usepackage[a4paper,textwidth=15cm,textheight=25cm]{geometry}
\usepackage{amssymb}
\usepackage{amsmath}
\usepackage{amsfonts}
\usepackage{bm}
\usepackage{authblk}
\usepackage{tikz}

\newtheorem{theorem}{Theorem}

\newtheorem{definition}[theorem]{Definition}

\newtheorem{remark}[theorem]{Remark}

\newenvironment{proof}[1][Proof]{\noindent\textbf{#1. }}{\hfill \rule{0.5em}{0.5em} \medskip\newline }

\newcommand{\bF}{\overline{F}}

\newcommand{\IFR}{{\rm c}}
\newcommand{\DFR}{{\rm DFR}}
\newcommand{\sIFR}[1]{\ensuremath{#1\!-\!}{\rm IFR}}

\newcommand{\IFRA}{\ensuremath{\ast}}
\newcommand{\DFRA}{{\rm DFRA}}
\newcommand{\sIFRA}[1]{\ensuremath{#1\!-\!}{\rm IFRA}}

\title{
On a conjecture about the comparability of parallel systems with respect to convex transform order\thanks{This work was partially supported by the Centre for Mathematics of the University of Coimbra -- UID/MAT/00324/2013, funded by the Portuguese Government through FCT/MEC and co-funded by the European Regional Development Fund through the Partnership Agreement PT2020.}}

\author[1]{Idir Arab}
\author[2]{Milto Hadjikyriakou}
\author[1]{Paulo Eduardo Oliveira}

\affil[1]{CMUC, Department of Mathematics, University of Coimbra, Portugal}
\affil[2]{University of Central Lancashire, Cyprus}

\date{}
\begin{document}

\maketitle

\begin{abstract}
In this paper we prove that two heterogeneous parallel systems with independent exponentially distributed components are comparable via the star transform order while the comparison via the convex transform fails. The latter conclusion provides a partial answer to a problem that remained open for a decade.

\smallskip


\end{abstract}

\section{Introduction}
Deciding about the ageing properties of systems whose lifetime is random requires an appropriate meaning of the comparison criterium. The literature is abundant in alternative definitions of ageing properties and the corresponding ordering of the lifetime distributions. Generally speaking, the approach starts by the definition of a way to measure a relevant risk. The most common risk notions are the reliability function, the conditional survival function, the failure rate or the expected value of residual life. There are, of course, several other notions that may be used to compare the behaviour of lifetime distributions. We will be interested in failure rate risk properties. The characterization of monotonicity of the failure rate function for a lifetime distribution is an important aspect and this has been studied, among others, by Barlow and Proschan~\cite{BP-81}, Patel~\cite{Pt-83}, Sengupta~\cite{Sp-94} and El-Bassiouny~\cite{ElB-03}.

Once we defined a risk to be measured, it is important to determine how lifetime distributions are ordered with respect to the given risk. This may be viewed as deciding which distribution is ageing faster with respect to the defined risk. These order relations usually define partial orderings, a subject that has been studied extensively by various authors (see for example, Desphande et al.~\cite{Des86}, Kochar and Wiens~\cite{KW-87}, Singh~\cite{SH-89}, Fagiuoli and Pellerey~\cite{FP93} or Shaked and Shanthikumar~\cite{SS07}). In most cases the order relations defined are included in some family of transform order relationships. We refer the reader to section 4.B in Shaked and Shanthikumar~\cite{SS07} for general definitions and properties of such notions. It is interesting to mention that one of the transform orders we will be studying later, the convex transform order, has a geometric interpretation, as described by van Zwet~\cite{VZ-70}, providing a way to compare the skeweness of lifetime distributions.

This note responds to the ageing order between parallel systems that has been conjectured in Kochar and Xu~\cite{KX09} who were interested in comparing the ageing performance of parallel systems with respect to the convex transform order mentioned above. In particular, they were interested in comparing the ageing properties of parallel systems whose components are independent and have exponential lifetime distributions. In their Theorem~3.1, they proved that, for systems with the same number of components, a parallel system with homogeneous components ages faster, with respect to the convex transform order, compared to a parallel system with heterogeneous components. In Remark~3.2 in Kochar and Xu~\cite{KX09}, it is conjectured that the same ageing behaviour holds when comparing two heterogeneous systems based on components that have exponentially distributed lifetimes
with hazard rates that, besides having the same sum, can be ordered in a suitable way (see Definition~\ref{def:ord} below for details).
The authors justify this conjecture based on the intuitive extension of their Theorem~3.1 and on the empirical evidence they collected. Of course, being a conjecture, Kochar and Xu were not able to give a mathematical proof for this. Our Theorem~\ref{thm:conj} below, proves that the conjecture is not valid. Furthermore, the proof provides an explanation for the empirical evidence Kochar and Xu collected, as it is shown that the choice of parameters that violates the convex transform ordering happens in a rather narrow region, which is easily missed if no prior indication is available.

\section{Preliminaries}
Let $X$ be a nonnegative random variable with density function $f_X$, distribution function $F_X$, and tail function $\overline{F}_X=1-F_X$. Moreover, for each $x\geq 0$ the failure rate function of $X$ is given by $r_{X}(x)=\frac{f_X(x)}{1-F_X(x)}$. Two of the most simple and common ageing notions are defined in terms of the failure rate function. Their definitions are given below. 
\begin{definition}
Let $X$ be a nonnegative valued random variable.
\begin{enumerate}
\item
$X$ is said {\rm IFR} (resp. \DFR) if $r_{X}$ is increasing (resp. decreasing) for $x\geq 0$.
\item
$X$ is said {\rm IFRA} (resp. $\DFRA$) if $\frac{1}{x}\int_0^x r_{X,s}(t)\,dt$ is increasing (resp. decreasing) for $x>0$.
\end{enumerate}
\label{def:s-fail}
\end{definition}
The above definitions refer to monotonicity properties of the distribution. In the following, we introduce criteria that will allow us to compare distribution functions.
\begin{definition}
\label{DEF S-IFR}
Let $\mathcal{F}$ denote the family of distributions functions such that $F(0)=0$, $X$ and $Y$ be nonnegative random variables with distribution functions $F_X,F_Y\in\mathcal{F}$.
    \begin{enumerate}
    \item
    The random variable $X$ (or its distribution $F_X$) is said smaller than $Y$ (or its distribution $F_Y$) in convex transform order, and we write $X\leq_{\IFR}Y$, or equivalently, $F_X\leq_{\IFR}F_Y$, if $\overline{F}_Y^{-1}(\overline{F}_X(x))$ is convex.
    \item
    The random variable $X$ (or its distribution $F_X$) is said smaller than $Y$ (or its distribution $F_Y$) in star transform order, and we write $X\leq_{\IFRA}Y$, or equivalently, $F_X\leq_{\IFRA}F_Y$, if $\frac{\overline{F}_Y^{-1}(\overline{F}_X(x))}{x}$ is increasing (this is also known as $\overline{F}_Y^{-1}(\overline{F}_X(x))$ being star-shaped).
    \end{enumerate}
\end{definition}
The convex transform order has been introduced by van Zwet~\cite{VZ-70} to characterize and order the skewness of densities of lifetime distributions, giving a definition for $X$ being less skewed than $Y$. The definitions above fall in the family of iterated {\rm IFR} and {\rm IFRA} orders, respectively, introduced and initially studied in Nanda et al.~\cite{ASOK}, Arab and Oliveira~\cite{AO18} or Arab, Hadjikyriakou and Oliveira~\cite{AMO18}, considering their iteration parameter to be 1. Indeed, it is immediate to verify that the convex order relation ``$\leq_\IFR$'' is denoted in the above mentioned references by ``$\leq_{\sIFR{1}}$'', while the star transform order ``$\leq_\IFRA$'' is denoted by ``$\leq_{\sIFRA{1}}$''. It is particularly useful to highlight at this point that Nanda et al.~\cite{ASOK} proved that the iterated {\rm IFR} and {\rm IFRA} orderings, that is the convex transform and the star transform orders, define partial order relations in the equivalence classes of $\mathcal{F}$ corresponding to the equivalence relation $F\sim G$ defined by $F(x)=G(kx)$, for some $k>0$. In case of families of distributions that have a scale parameter, this allows to choose the parameter in the most convenient way. 

A general characterization of the above transform order relations is given below (see Propositions~3.1 and 4.1 in Nanda et al.~\cite{ASOK}).
\begin{theorem}
\label{convexity-equivalence}
Let $X$ and $Y$ be random variables with distribution functions $F_X,F_Y\in\mathcal{F}$.
    \begin{enumerate}
    \item
    $X\leq_{\IFRA}Y$ if and only if for any real number $a$, $\bF_{Y}(x)-\bF_{X}(ax)$ changes sign at most once, and if the change of signs occurs, it is in the order ``$-,+$'', as $x$ traverses from $0$ to $+\infty$.
    \item
    $X\leq_{\IFR}Y$ if and only if for any real numbers $a$ and $b$, $\bF_{Y}(x)-\bF_{X}(ax+b)$ changes sign at most twice, and if the change of signs occurs twice, it is in the order ``$+,-,+$'', as $x$ traverses from $0$ to $+\infty$.
    \end{enumerate}

\begin{remark}
\label{AOalpha}
As mentioned in Remark 25 in Arab and Oliveira~\cite{AO18}, it is enough to verify the above characterizations for $a>0$.
\end{remark}
\end{theorem}
The characterization given by Theorem~\ref{convexity-equivalence} requires explicit expressions of the tails of the distributions, which are often not available. Computationally tractable characterizations to decide about the actual comparison of general distributions were studied in Arab and Oliveira~\cite{AO18} and Arab, Hadjikyriakou and Oliveira~\cite{AMO18} (see, Theorems~2.3 and 2.4 in the later reference). As one may verify in the proofs given in \cite{AO18} when proving the convex transform order, the control of the sign variation is usually more complex when considering $b<0$. However, a prior verification of the star transform ordering may help circumventing this difficulty, as expressed by the result quoted below, which is a reduced version, adapted to the present framework of Theorem~29 in Arab, Hadjikyriakou and Oliveira~\cite{AMO18}.
\begin{theorem}
\label{thm:newcrit}
Let $X$ and $Y$ be random variables with distribution functions $F_X,F_Y\in\mathcal{F}$, respectively. If $X\leq_{\IFRA}Y$ and the criterium in \textit{2.} from Theorem~\ref{convexity-equivalence} is verified for $b\geq0$, then $X\leq_{\IFR}Y$.
\end{theorem}


The lifetime of parallel systems is expressed as the maximum of the lifetimes of each component. When these components have exponentially distributed lifetimes, the distribution functions of the system's lifetime is expressed as a linear combination of exponential terms. Later, it will be important to be able to count and localize the roots of such expressions. The following result will play an important role on this aspect (see Tossavainen~\cite{Toss07}, or Theorem~1 in Shestopaloff~\cite{She11}).
\begin{theorem}
\label{thm:zeros}
Let $n\geq 0$, $p_0>p_1>\cdots>p_n>0$, and $\alpha_j\ne0$, $j=0,1,\ldots,n$, be real numbers. Then the function $f(t)=\sum_{j=0}^n\alpha_jp^t_j$ has no real zeros if $n=0$, and for $n\geq 1$ has at most as many real zeros as there are sign changes in the sequence of coefficients $\alpha_0,\alpha_1,\alpha_2,\ldots,\alpha_n$.
\end{theorem}

\section{Main results}
We begin this section by quoting, for the sake of completeness, Theorem~3.1 by Kochar and Xu~\cite{KX09}. It was this result that suggested the conjecture we will be discussing in the sequel.

\begin{theorem}
\label{thm:KX09}
Let $X_1,\ldots,X_n$ be independent random variables with exponential distribution with common hazard rate $\lambda$. Let $Y_1,\ldots,Y_n$ be independent random variables with exponential distribution with hazard rates $\theta_i$, $i=1,\ldots,n$. Then $\max(X_1,\ldots,X_n)\leq_\IFR\max(Y_1,\ldots,Y_n)$.
\end{theorem}

\begin{remark}
This order relation has been extended, for the case $n=2$, by Theorem~7.2 in Arab, Hadjikyriakou and Oliveira~\cite{AMO18}, considering the iterated order relations mentioned above, which extend the convex transform order.
\end{remark}

For the remainder of this section, we will be interested in characterizing the order relationship between parallel systems of heterogeneous components with exponential lifetime distributions. We recall an order relation between vectors of hazard rates of each component introduced in Definition A.1 in Marshall and Olkin~\cite{MO79}.
\begin{definition}
\label{def:ord}
Let $(\lambda_1,\ldots,\lambda_n),\,(\theta_1,\ldots,\theta_n)\in \mathbb{R}^{n}$ two vectors such that $\lambda_1\leq\cdots\leq\lambda_n$ and $\theta_1\leq\cdots\leq\theta_n$. 
We say that $(\lambda_1,\ldots,\lambda_n)\prec (\theta_1,\ldots,\theta_n)$ if
$$
\sum_{i=1}^{k}\lambda_i \geq \sum_{i=1}^{k}\theta_i\quad\mbox{for}\quad k=1,\ldots,n-1,\qquad\mbox{and}\qquad \sum_{i=1}^{n}\lambda_i = \sum_{i=1}^{n}\theta_i.
$$
\end{definition}
Kochar and Xu~\cite{KX09} conjectured that the conclusion of Theorem~\ref{thm:KX09} would still hold if the $X_i$ have hazard rate $\lambda_i$ and the $Y_i$ have hazard $\theta_i$ such that $(\lambda_1,\ldots,\lambda_n)\prec(\theta_1,\ldots,\theta_n)$. We will be proving below that this conjecture only holds for systems with two components if we replace the order relation ``$\leq_\IFR$'' by ``$\leq_\IFRA$''.

\begin{theorem}
\label{thm:ifra}
Let $X_1$ and $X_2$ be independent random variables with exponential distributions with hazard rates $\lambda_1$ and $\lambda_2$, respectively. Analogously, let $Y_1$ and $Y_2$ be independent random variables with exponential distributions with hazard rates $\theta_1$ and $\theta_2$, respectively. If $(\lambda_1,\lambda_2)\prec(\theta_1,\theta_2)$, then $X=\max(X_1,X_2)\leq_{\IFRA}Y=\max(Y_1,Y_2)$.
\end{theorem}
\begin{proof}
Let $\bF_X$ and $\bF_Y$ be the survival functions of $X$ and $Y$, respectively. Then we have
$$
\bF_X(x)=e^{-\lambda_1x}+e^{-\lambda_2x}-e^{-(\lambda_1+\lambda_2)x}\qquad\mbox{and}\qquad
\bF_Y(x)= e^{-\theta_1x}+e^{-\theta_2x}-e^{-(\theta_1+\theta_2)x}.
$$
Taking into account Theorem~\ref{convexity-equivalence} and Remark~\ref{AOalpha}, it is sufficient to prove that $V(x)=\bF_Y(x)-\bF_X(ax)$ changes sign at most once, and if the sign change occurs, it is in the order ``$-,+$'', when $x$ traverses the interval $[0,+\infty)$, for every real number $a>0$. We will consider three separate cases, depending on the value of $a$. First, note that the assumption $(\lambda_1,\lambda_2)\prec(\theta_1,\theta_2)$ implies that $\theta_1\leq\lambda_1\leq\lambda_2\leq\theta_2$ and $\lambda_1+\lambda_2=\theta_1+\theta_2$.
\begin{description}
\item[{\rm\textit{Case 1}. $a=1$:}]
The function $V$ is rewritten now as $V(x)=e^{-\theta_1x}+e^{-\theta_2x}-(e^{-\lambda_1x}+e^{-\lambda_2x})$. Reordering the exponential terms so that they are appear in decreasing order of their basis, the sign pattern of the coefficientsis ``$+,-,-,+$''. Hence, according to Theorem~\ref{thm:zeros}, $V$ has at most two real roots. Moreover, $\lim_{x\rightarrow-\infty}V(x)=+\infty$, while $\lim_{x\rightarrow+\infty}V(x)=0^+$. Furthemore, taking into account that $V(0)=0$, $V^{\prime}(0)=0$ and $V^{\prime\prime}(0)=-\lambda_1^2+\theta_1^2-\lambda_2^2+\theta_2^2=-(\theta_2-\lambda_2)(\lambda_1-\lambda_2+\theta_1-\theta_2)>0$, it follows that $V(x)\geq0$, which means that $\bF_Y(x)\geq\bF_X(x)$, for every $x\in\mathbb{R}$, thus no sign changes occur.

\item[{\rm\textit{Case 2}. $a>1$:}] As $\bF_X$ is decreasing, it follows that, for $x\geq0$, $V(x)\geq \bF_Y(x)-\bF_X(x)\geq0$ so, again, no sign changes occur.

\item[{\rm\textit{Case 3}. $0<a<1$:}] To analyse the sign pattern of the coefficients in $V$, we distinguish two cases:
    \begin{description}
      \item[{\rm $\theta_1<a\lambda_1$:}]
      After reordering the exponentials in $V$ according the their basis, the sign pattern of the coefficients is ``$+,-,-,+,+,-$''. Thus, according to Theorem~\ref{thm:zeros}, $V$ has at most three real roots. The sign pattern of the coefficients implies that $\lim_{x\rightarrow-\infty}V(x)=-\infty$, while $\lim_{x\rightarrow+\infty}V(x)=0^+$. Finally, taking into account that $V(0)=V^{\prime}(0)=0$, the possible sign changes for $V$ when $x$ traverses from 0 to $+\infty$ are either ``$+$'' or ``$-,+$''.
      \item[{\rm $\theta_1\geq a\lambda_1$:}]
      As $\bF_X$ is decreasing and $a<\frac{\theta_1}{\lambda_1}$, it follows that, for $x\geq0$,
      $$
      V(x)\leq H(x)=\bF_Y(x)-\bF_X(\tfrac{\theta_1}{\lambda_1}x)
       =e^{-\theta_2 x}-e^{-(\theta_1+\theta_2)x}
          -\left(e^{-\frac{\theta_1\lambda_2}{\lambda_1}x}-e^{-\frac{\theta_1}{\lambda_1}(\lambda_1+\lambda_2)x}\right).
      $$
      After reordering the exponentials in $H$ according to their basis, the sign pattern of the coefficients in $H$ is ``$-,+,+,-$'', implying that, according to Theorem~\ref{thm:zeros}, $H$ has at most two real roots. The sign of the coefficients of $H$ also implies that 
      $\lim_{x\rightarrow+\infty}H(x)=0^-$ which, together with the fact that $H(0)=H^\prime(0)=0$ and $H^{\prime\prime}(0)=2\theta_1(\frac{\theta_1}{\lambda_1}\lambda_2-\theta_2)<0$, further imply that $H(x)\leq0$, so consequently $V(x)\leq0$, that is, no sign changes occur.
    \end{description}
  \end{description}
So, finally, $V$ has at most one sign change when $x$ goes from 0 to $+\infty$ and, if the change occurs, it is in the order ``$-,+$''. Thus, according to Theorem~\ref{convexity-equivalence}, $X\leq_\IFRA Y$.
\end{proof}

\begin{theorem}
\label{thm:conj}
Let $X$ and $Y$ be as in Theorem~\ref{thm:ifra}. Then $X$ and $Y$ are not comparable with respect to convex transform order.
\end{theorem}
\begin{proof}
We will start the discussion by showing the control of the sign variation of $V(x)=\bF_Y(x)-\bF_X(ax+b)$ when possible, exhibiting the choice of parameters $a>0$ and $b>0$ where the control is not possible and describing a way to find parameters where, indeed, the sign variation violates the criterium given in \textit{2.} of Theorem~\ref{convexity-equivalence}. Note that, unlike when studying the star transform order case, we have that $V(0)=1-\bF_X(b)>0$.
\begin{description}
\item[{\rm\textit{The favorable cases}.}]
Taking into account Theorems~\ref{convexity-equivalence} and \ref{thm:newcrit}, and Remark~\ref{AOalpha} we can achieve the appropriate sign control in the cases described below, depending only on the value of $a>0$ and  $b>0$. 
The control on the sign variation is, in each case, obtained by the identification of the possible number of real roots with Theorem~\ref{thm:zeros}, and coupling this with the behaviour of $V$ when $x\longrightarrow\pm\infty$.
    \begin{description}
    \item[{\rm\textit{Case 1}. $a\geq 1$:}]
    As both parameters are nonnegative and we are interested in $x\geq0$, we have $ax+b\geq x$, hence $V(x)\geq\bF_Y(x)-\bF_X(x)\geq 0$.
    \item[{\rm\textit{Case 2}. $\frac{\theta_1}{\lambda_1}\leq a<1$:}]
    After reordering the appropriately the terms in $V$ the sign pattern of its coefficients is ``$+,-,-,+,+,-$'' (or ``$+,-,+,+,-$'' if $a=\frac{\theta_1}{\lambda_1}$), hence, according to Theorem~\ref{thm:zeros}, $V$ has at most three real roots. Moreover, from the signs of the coefficients it follows that 
    $\lim_{x\rightarrow+\infty}V(x)=0^+$ (in both cases). As $V(0)>0$, this means there are at most two nonnegative real roots.
    Thus, when $x$ traverses from 0 to $+\infty$, the sign pattern can only be either ``$+$'' or ``$+,-,+$''.
    \item[{\rm\textit{Case 3}. $0<a\leq\frac{\theta_1}{\lambda_2}<1$:}]
    Reordering again the terms in $V$ to apply Theorem~\ref{thm:zeros}, we find the sign pattern for its coefficients ``$-,-,+,+,+,-$'' (collapsing to `$-,+,+,+,-$'' if $a=\frac{\theta_1}{\lambda_2}$), so $V$ has at most two real roots. At infinity, we find that
    $\lim_{x\rightarrow+\infty}V(x)=0^-$. Hence, as $V (0)>0$, $V$ has one nonnegative real root and the sign pattern of $V$ when $x$ goes from 0 to $+\infty$ is ``$+,-$''.
    \end{description}

\item[{\rm\textit{The violating case}. $\frac{\theta_1}{\lambda_2}<a<\frac{\theta_1}{\lambda_1}$:}]
Applying Theorem~\ref{thm:zeros} to the present case, we find the following sign pattern for the coefficients of $V$: ``$-,+,-,+,+,-$''. Thus, we derive that there are at most four real roots. As, also from the sign of the coefficients of $V$, it follows that $\lim_{x\rightarrow-\infty}V(x)=-\infty$, which together with the fact and $V(0)>0$, implies that one of the roots is negative. Again, from the signs of the coefficients of $V$, it follows that $\lim_{x\rightarrow+\infty}V(x)=0^-$, and this is compatible with the sign variations, when $x$ traverses from 0 to $+\infty$, ``$+,-$'' or ``$+,-,+,-$''. That is, the usage of Theorem~\ref{thm:zeros} is not conclusive...

We need a different approach to show that the sign variation ``$+,-,+,-$'' is indeed achieved for an appropriate choice of the parameters $a>0$ and $b>0$, hence violating the comparison criterium. From the previous analysis, we know that if $a=\frac{\theta_1}{\lambda_2}$ the sign variation of $V$ as $x$ traverses from 0 to $+\infty$ is ``$+,-$''. Likewise, we also know that if $a=\frac{\theta_1}{\lambda_1}$ the sign variation of $V$ as $x$ goes from 0 to $+\infty$ is either ``$+$'' or ``$+,-,+$''. The actual verification of each the later possible sign variations may be achieved by a suitable choice of $b>0$. Let us choose $b_0>0$ such that the sign variation of $V(x)=\bF_Y(x)-\bF_X(\frac{\theta_1}{\lambda_1}x+b_0)$ when $x$ goes from 0 to $+\infty$ is ``$+,-,+$'', and keep this choice fixed for the sequel of the proof. Furthermore, remember that when $\frac{\theta_1}{\lambda_2}<a<\frac{\theta_1}{\lambda_1}$ we have verified that $\lim_{x\rightarrow+\infty}V(x)=0^-$. Hence we have the following graphical description of the sign of $V$, depending on $x$ and $a$:
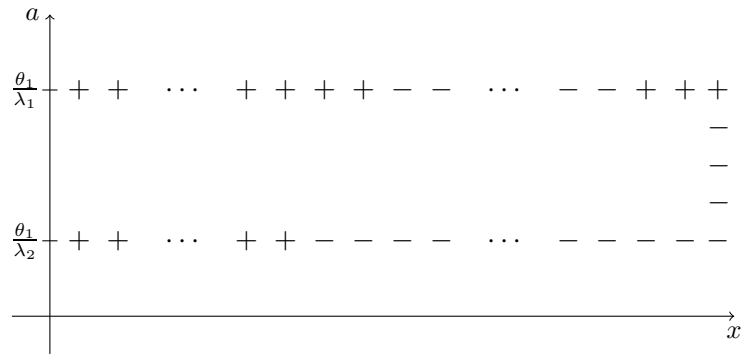
\begin{center}
\begin{tikzpicture}
\draw[->] (-.5,0) -- (9,0) node [below] {\small $x$};
\draw[->] (0,-.5) -- (0,4.0) node [left] {\small $a$};
\draw (-.1,1) -- (.1,1) node [left] {\small $\frac{\theta_1}{\lambda_2}\;$} node [right] {$+\ +\quad\cdots\quad +\ +\ -\ -\ -\ -\quad\cdots\quad -\ -\ -\ -\ -$};
\draw (-.1,3) -- (.1,3) node [left] {\small $\frac{\theta_1}{\lambda_1}\;$} node [right] {$+\ +\quad\cdots\quad +\ +\ +\ +\ -\ -\quad\cdots\quad -\ -\ +\ +\ +$};
\draw(8.8,1.5) node {$-$};
\draw(8.8,2.0) node {$-$};
\draw(8.8,2.5) node {$-$};
\end{tikzpicture}
\end{center}
So far, to analyse the sign variation of $V$, we have been looking at the function as depending on $x$ alone. The previous argument suggests viewing $V$ as a function of $x$, $a$ and $b$, although we will keep $b=b_0$ fixed. Thus, we may differentiate with respect to $a$, to find
$$
\frac{\partial V}{\partial a}(x,a,b)  =
    x\left(\lambda_1 e^{-\lambda_1(ax+b)}+\lambda_2 e^{-\lambda_2(ax+b)}-(\lambda_1+\lambda_2) e^{-(\lambda_1+\lambda_2)(ax+b)}\right)
  =  xf_X(ax+b),
$$
where $f_X$ is the density function of $X$ which implies that for every possible choice for $b>0$, in particular for $b=b_0$, $\frac{\partial V}{\partial a}(x,a,b)>0$. Hence, as a function of $a$ alone, $V$ is increasing. Thus, when $a$ increases from $\frac{\theta_1}{\lambda_2}$ to $\frac{\theta_1}{\lambda_1}$ the value for $V$ is also increasing, implying that once it becomes positive it may no longer get back to negative values. For the particular choice $b=b_0$, that produces the lines of signs ``$+$'' and ``$-$'' above, the increasingness of $V$ with respect to $a$ explains why the initial sequence of ``$+$'' signs for $a_0=\frac{\theta_1}{\lambda_1}$ is longer than the corresponding initial sequence when $\frac{\theta_1}{\lambda_2}$. It remains to verify that such choice of $b_0$ does exist. First, note that we proved in Case 3. of the proof of Theorem~\ref{thm:ifra} that $\bF_Y(x)-\bF_X(a_0x)\leq 0$, for every $x\geq 0$, and the inequality is strict for every $x>0$. Now, choosing some $x_0>0$, we have that $\bF_X^{-1}(\bF_Y(x_0))>a_0x_0$, so we may find $b_0$ (depending on $x_0$) such that $\bF_X^{-1}(\bF_Y(x_0))>a_0x_0+b_0$, which implies that $\bF_Y(x_0)<\bF_X(a_0x_0+b_0)$. As the functions are continuous, this inequality will hold on some neighbourhood of $x_0$, so the sign pattern represented above always happens.

Getting back to the graphical representation above, we now locate the line of points such that $V(x,a,b_0)=0$. As the functions are continuous, this line is also continuous, and we will find the behaviour described by the thick line below, where we also identify the sign of $V$ in each region:
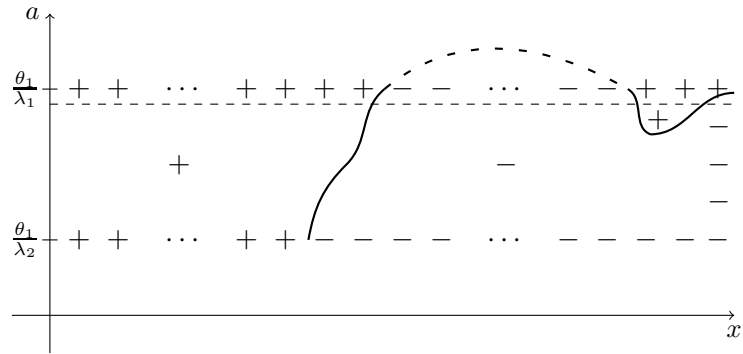
\begin{center}
\begin{tikzpicture}
\draw[->] (-.5,0) -- (9,0) node [below] {\small $x$};
\draw[->] (0,-.5) -- (0,4.0) node [left] {\small $a$};
\draw (-.1,1) -- (.1,1) node [left] {\small $\frac{\theta_1}{\lambda_2}\;$} node [right] {$+\ +\quad\cdots\quad +\ +\ -\ -\ -\ -\quad\cdots\quad -\ -\ -\ -\ -$};
\draw (-.1,3) -- (.1,3) node [left] {\small $\frac{\theta_1}{\lambda_1}\;$} node [right] {$+\ +\quad\cdots\quad +\ +\ +\ +\ -\ -\quad\cdots\quad -\ -\ +\ +\ +$};
\draw (8.8,1.5) node {$-$};
\draw (8.8,2.0) node {$-$};
\draw (8.8,2.5) node {$-$};
\draw[thick] (3.4,1) to [out=80,in=225] (3.9,2) to [out=45,in=220] (4.4,3);
\draw[loosely dashed, thick] (4.4,3) to [out=40,in=150] (7.6,3);
\draw[thick] (7.6,3) to [out=330,in=160] (7.9,2.4) to [out=0,in=180] (9,2.95);
\draw (1.7,2) node {$+$};
\draw (8.0,2.6) node {$+$};
\draw (6,2) node {$-$};
\draw[dashed] (0,2.8) -- (9,2.8);
\end{tikzpicture}
\end{center}
\end{description}
The position of the horizontal dashed line identifies a value for the parameter $a$ for which the sign variation of $V$, with $b=b_0$, is actually ``$+,-,+,-$'', so the random variables are not comparable with respect to the convex transform order.
\end{proof}
\begin{remark}
The construction above really depends on an appropriate choice for $b$, producing the ``$+,-,+$'' sign pattern on the top line. Numerical experiments indicate that this may be achieved by choosing $b$ close to 0. This behaviour for the sign patterm is critical for the argument used above. Moreover, numerical experiments suggest that, even with a convenient choice for $b$ the top-right region with the ``$+$'' signs may be relatively narrow. Without prior indication of where to look for, it is easy to miss the appropriate choice for $a$ and $b$.

The original motivation for the conjecture stems from the characterization of skewness of the densities, in the sense introduced by van Zwet~\cite{VZ-70}, which was the problem studied by Kochar and Xu~\cite{KX09}. The link between the convexity approach used by Kochar and Xu~\cite{KX09} and our sign variation approach follows from the characterization of convexity by means of sign variation, described in Theorem~20 in Arab and Oliveira~\cite{AO18}. The narrow region for the choice of $a$ and $b$ violating the sign variation pattern, means that the function $\bF_Y^{-1}(\bF_X(x))$ is convex in almost the whole of its domain and concave in a small interval. 
\end{remark}
An explicit example and choice of the parameters violating the $\leq_\IFR$-comparability may be obtained taking $\lambda_1=2$, $\lambda_2=3$, $\theta_1=1.5$, $\theta_2=3.5$, $a=0.749$ and $b=0.0125$. One will still need to use a lot of zooming to actually see the sign changing behaviour as $x$ grows.
\begin{remark}[An open problem]
Finally, we note that the conjecture by Kochar and Xu~\cite{KX09} refers to parallel systems with $n\geq 2$ components each. Our Theorem~\ref{thm:conj} resolves the conjecture for $n=2$. For $n>2$, the comparison between $X=\max(X_1,\ldots,X_n)$, where the $X_i$ are independent exponentially distributed with hazard rates $\lambda_i$, and $Y=\max(Y_1,\ldots,Y_n)$, where the $Y_i$ are also independent and exponentially distributed with hazard rates $\theta_i$, remains open, even with respect to the star transform order. Again, numerical experiments suggest that $X\leq_\IFRA Y$ may hold whenever $(\lambda_1,\ldots,\lambda_n)\prec(\theta_1,\ldots,\theta_n)$, while these random variables seem not to be comparable with respect to the convex transform order.
\end{remark}


\begin{thebibliography}{99}



\bibitem{AO18}
\textsc{Arab, I. and Oliveira, P. E.} (2018). Iterated Failure Rate Monotonicity And Ordering Relations Within Gamma And Weibull Distributions. To appear in \emph{Probab. Eng. Inform. Sc.} doi:10.1017/S0269964817000481

\bibitem{AMO18}
\textsc{Arab, I., Hadjikyriakou, M. and Oliveira, P. E.} (2018). Failure rate properties of parallel systems. Preprint, Pr\'{e}-Publica\c{c}\~{o}es do Departamento de Matem\'{a}tica da Universidade de Coimbra, 18-51, 2018



\bibitem{BP-81}
\textsc{Barlow, R. E. and Proschan, F.} (1975). \emph{Statistical theory of reliability and life testing. Probability models}. Holt, Rinehart and Winston, Inc., New York-Montreal, Que.-London.







%


%
%
%
%




\bibitem
{Des86}
\textsc{Deshpande, J. V., Kochar, S. C. and Singh, H.} (1986). Aspects of positive ageing. \emph{J. Appl. Probab.}~\textbf{23,} 748--758. doi: 10.2307/3214012




\bibitem{ElB-03}
\textsc{El-Bassiouny, A. H.} (2003). On testing exponentiality against IFRA alternatives. \emph{Appl. Math. Comput.}~\textbf{146,} 445--453. doi: 10.1016/S0096-3003(02)00597-0

\bibitem
{FP93}
\textsc{Fagiuoli, E. and Pellerey, F.} (1993). New partial orderings and applications. \emph{Nav. Res. Log.}~\textbf{40,} 829--842. doi: 10.1002/1520-6750(199310)40:6$<$829::AID-NAV3220400607$>$3.0.CO;2-D






\bibitem{KW-87}
\textsc{Kochar, S. C. and Wiens, D. D.} (1987). Partial orderings of life distributions with respect to their ageing properties. \emph{Nav. Res. Log.}~\textbf{34,} 823--829. doi: 10.1002/1520-6750(198712)34:6$<$823::AID-NAV3220340607$>$3.0.CO;2-R




\bibitem{KX09}
\textsc{Kochar, S. C. and Xu, M.} (2009). Comparisons of parallel systems according to the convex transform order. \emph{J. Appl. Probab.}~\textbf{46,} 342--352. doi: 10.1239/jap/1245676091





\bibitem{MO79}
\textsc{Marshall, A. W. and Olkin, I.} (1979). \emph{Inequalities: Theory of Majorization and Its Application}. Academic Press, NewYork.



\bibitem
{ASOK}
\textsc{Nanda, A. K., Hazra, N. K., Al-Mutairi, D.K. and Ghitany, M.E.} (2017). On some generalized ageing orderings. \emph{Commun. Stat. Theory Methods}~\textbf{46,} 5273--5291. doi: 10.1080/03610926.2015.1100738



\bibitem{Pt-83}
\textsc{Patel, J. K.} (1983). Hazard rate and other classifications of distributions. In \emph{Encyclopedia in Statistical Sciences}~\textbf{3,} John Wiley \& Sons, pp. 590--594. doi: 10.1002/0471667196.ess0935.pub2







\bibitem{Sp-94}
\textsc{Sengupta, D.} (1994). Another look at the moment bounds on reliability. \emph{J. Appl. Probab.}~\textbf{31,} 777--787. doi: 10.2307/3215155

\bibitem
{SS07}
\textsc{Shaked, S. and Shanthikumar, J. G.} (2007). \emph{Stochastic orders}. Springer, New York.

\bibitem{SH-89}
\textsc{Singh, H.} (1989). On partial orderings. \emph{Nav. Res. Log.}~\textbf{36,} 103–-110. doi: 10.1002/1520-6750(198902)36:1$<$103::AID-NAV3220360108$>$3.0.CO;2-7



\bibitem{She11}
\textsc{Shestopaloff, Yu. K.} (2011). Properties of Sums of Some Elementary Functions and Their Application to Computational and Modeling Problems. \emph{Comput. Math. Math. Phys.}~\textbf{51,} 699--712. doi: 10.1134/S0965542511050162

\bibitem{Toss07}
\textsc{Tossavainen, T.} (2007). The Lost Cousin of the Fundamental Theorem of Algebra. \emph{Maths. Magazine}~\textbf{80,} 290--294. doi: 10.1080/0025570X.2007.11953496

\bibitem{VZ-70}
\textsc{van Zwet, W. R.} (1964). Convex transformations of random variables. MC Tracts~\textbf{7}, Amsterdam.



\end{thebibliography}
\end{document}